\documentclass[12pt]{article}

\usepackage{amsmath,amsfonts,amssymb}
\usepackage{bbm}

\usepackage{kotex}
\usepackage{lipsum}

\usepackage{xcolor}
\usepackage{ulem}

\newtheorem{defn}{Definition}[section]
\newtheorem{theo}[defn]{Theorem}
\newtheorem{lem}[defn]{Lemma}
\newtheorem{prop}[defn]{Proposition}

\newtheorem{rem}[defn]{Remark}
\newtheorem{exam}[defn]{Example}

\newcommand\blfootnote[1]{%
  \begingroup
  \renewcommand\thefootnote{}\footnote{#1}%
  \addtocounter{footnote}{-1}%
  \endgroup
}

\newenvironment{proof}{{\bf Proof }}{{\vskip 0.1cm \hfill$\Box$}}

\begin{document} 
\noindent
{\Large \bf Strong uniqueness of finite dimensional Dirichlet \\[2pt]operators with singular drifts}
{\blfootnote{The research of Haesung Lee was supported by Basic Science Research Program through the National Research Foundation of Korea (NRF) funded by the Ministry of Education (2020R1A6A3A01096151). \\[5pt]
  }} \\ \\
\bigskip
\noindent
{\bf Haesung Lee}  \\
\noindent
{\small{\bf Abstract.}  
We show $L^r(\mathbb{R}^d, \mu)$-uniqueness  for any $r \in (1, 2]$ and the essential self-adjointness of a Dirichlet operator $Lf = \Delta f +\langle \frac{1}{\rho}\nabla \rho , \nabla f \rangle$,\, $f \in C_0^{\infty}(\mathbb{R}^d)$ with $d \geq 3$ and $\mu=\rho dx$. In particular, $\nabla \rho$ is allowed to be in $L^d_{loc}(\mathbb{R}^d, \mathbb{R}^d)$ or in $L^{2+\varepsilon}_{loc}(\mathbb{R}^d, \mathbb{R}^d)$ for some $\varepsilon>0$, while $\rho$ is required to be locally bounded below and above by strictly positive constants. The main tools in this paper are elliptic regularity results for divergence and non-divergence type operators and basic properties of Dirichlet forms and their resolvents.
\\

\section{Introduction} \label{intro}
The main target of this paper is to show $L^{r}(\mathbb{R}^d, \mu)$-uniqueness of a Dirichlet operator $(L, C_0^{\infty}(\mathbb{R}^d))$ for any $r \in (1, 2]$, where $\mu=\rho dx$ and
\begin{equation} \label{dirichletop}
Lf =  \Delta f +\langle \frac{1}{\rho} \nabla \rho, \nabla f \rangle, \qquad f \in C_0^{\infty}(\mathbb{R}^d)
\end{equation}
under the assumption {\bf (H)} or {\bf (S)} below. Here, the meaning of $L^r(\mathbb{R}^d, \mu)$-uniqueness of  $(L, C_0^{\infty}(\mathbb{R}^d))$ is that there exists only one $C_0$-semigorup of contractions on $L^r(\mathbb{R}^d, \mu)$ whose generator extends $(L, C_0^{\infty}(\mathbb{R}^d))$. $L^r(\mathbb{R}^d, \mu)$-uniqueness is also called strong uniqueness. Since $(L, C_0^{\infty}(\mathbb{R}^d))$ is a symmetric and positive operator in $L^2(\mathbb{R}^d, \mu)$, it follows from \cite[Lemma 1.4]{Eb} that $L^2(\mathbb{R}^d, \mu)$-uniqueness of $(L, C_0^{\infty}(\mathbb{R}^d))$ is equivalent to the essential self-adjointness of $(L, C_0^{\infty}(\mathbb{R}^d))$, i.e. the closure of $(L, C_0^{\infty}(\mathbb{R}^d))$ on $L^2(\mathbb{R}^d, \mu)$ is a self-adjoint operator. \\
The question for the essential self-adjointness of $(L, C_0^{\infty}(\mathbb{R}^d))$ is known as a classical problem in mathematical physics, because there exists a unitary equivalence between $L$ and a Schr\"{o}dinger operator $H=\Delta + V$ such that $H=\phi L \phi^{-1}$ with $\phi = \sqrt{\rho}$ and $V=\frac{\Delta \phi}{\phi}$. Since this equivalence is satisfied if $\rho$ is sufficiently regular, $L$ is sometimes called the generalized Schr\"{o}dinger operator. We refer to \cite{W85} and \cite{AHS77} for relevant details.\\
 Meanwhile, it is well-known, see \cite{T92} (and also \cite{RZ94}, \cite{CF96}), that if $\rho>0$ a.e. and $\sqrt{\rho} \in H^{1,2}_{loc}(\mathbb{R}^d)$, then $(L, C_0^{\infty}(\mathbb{R}^d))$ is Markov-unique, i.e. there exists only one symmetric and sub-Markovian $C_0$-semigroup of contractions on $L^2(\mathbb{R}^d, \mu)$ whose generator extends $(L, C_0^{\infty}(\mathbb{R}^d))$. 
But in order to obtain the essential self-adjointness of $(L, C_0^{\infty}(\mathbb{R}^d))$ using the existing literature, more regularity assumptions on $\rho$ are required.\\
Let us briefly introduce previous results on the essential self-adjointness of $(L, C_0^{\infty}(\mathbb{R}^d))$ progressed so far. N. Wielens showed in \cite[3.1 Theorem]{W85} that if 
$$
d \geq 2 \text{ and } \rho \text{ is locally Lipschitz continuous and strictly positive},
$$
then $(L, C_0^{\infty}(\mathbb{R}^d))$ is essentially self-adjoint. In particular, in case $d=1$, it turns out in \cite[2.3 Theorem]{W85} that $(L, C_0^{\infty}(\mathbb{R}))$ is essentially self-adjoint if 
$$
\text{ $\rho$ is absolutely continuous and strictly positive with $\rho' \in L_{loc}^2(\mathbb{R}, \mu)$.} 
$$
Regarding a global regularity assumption on $\rho$, it is known in \cite{LS92} that if $\frac{\nabla \rho}{\rho} \in L^4(\mathbb{R}^d, \mu)$, then $(L, C_0^{\infty}(\mathbb{R}^d))$ is essentially self-adjoint. On the other hand, pioneering work for essential self-adjointness of $(L, C_0^{\infty}(\mathbb{R}^d))$
was developed by V. Bogachev, N. V. Krylov and M. R\"{o}ckner under a weaker assumption of local regularity than the assumption of local Lipschitz continuity on $\rho$. They showed in \cite[Theorem 7]{BKR97} that $(L, C_0^{\infty}(\mathbb{R}^d))$ is essentially self-adjoint if the following condition {\bf ($\alpha$)} is assumed. \\
\text{}\\
{\it 
{\bf ($\alpha$)}{\rm :}\;  $d \geq 2$, $\rho=\varphi^2$ with $\varphi \in H^{1,2}_{loc}(\mathbb{R}^d)$ such that $\rho$ is locally uniformly positive, i.e. for any open ball $B$, $\inf_B \rho >0$. For some $p>d$, $\|\frac{\nabla \rho}{\rho}\| \in L^{p}_{loc}(\mathbb{R}^d, \mu)$.
}
\\
\\
Actually, it follows from \cite[Corollary 8]{BKR97} (cf. proofs of \cite[Theorem 7, Corollary 8]{BKR97})
that the condition {\bf ($\alpha$)} is
equivalent to 
\\
\\
{
\it
{\bf ($\alpha^\prime$)}{\rm :} \; $d \geq 2$, $\rho \in H^{1,p}_{loc}(\mathbb{R}^d) \cap C(\mathbb{R}^d)$ for some $p>d$ with $\rho(x)>0$ for all $x \in \mathbb{R}^d$.
}
\\
\\
Strong uniqueness for non-symmetirc operators is investigated in \cite{Eb20}, but the results there are covered by \cite{BKR97} 
if $L$ is as in \eqref{dirichletop} and $\rho$ is locally bounded and locally uniformly positive. By \cite[Theorem 1]{L99}, the following condition {\bf ($\beta$)} implies the essential self-adjointness of $(L, C_0^{\infty}(\mathbb{R}^d))$.
\\
\\
{
\it
{\bf ($\beta$)}{\rm :}
$d \geq 1$, $\|\frac{\nabla \rho}{\rho}\| \in L^4_{loc}(\mathbb{R}^d,\mu)$ and $R \in (0, \infty)$ is given. For any $R_1>R$, there exists a positive constant $C=C(R_1)$ such that the following
inequality holds}
\begin{equation*} \label{lavbasin}
\int_{B_{R_1} \setminus B_{R}} \Big \|\frac{1}{\rho}\nabla \rho \Big \|^2 \varphi^2 d\mu \leq  C \left( \|\nabla \varphi \|_{L^2(\mathbb{R}^d, \mu)} + \|\varphi\|_{L^2(\mathbb{R}^d, \mu)}  \right), \qquad \forall \varphi \in C_0^{\infty}(\mathbb{R}^d).
\end{equation*}
\text{}\\
If $d \geq 3$ and $\rho$ is locally bounded and locally uniformly positive and there exists $R>0$ such that $\|\nabla \rho\| \in L_{loc}^d(\mathbb{R}^d \setminus B_{R}) \cap L^4_{loc}(\mathbb{R}^d)$, then {\bf ($\beta$)} follows by the  H\"{o}lder  and Sobolev inequalities. Therefore, \cite[Theorem 1]{L99} can be regarded as a generalization of \cite[Theorem 7]{BKR97} in case $d \geq 4$. \\
On the other hand, M. Kolb developed in \cite[Theorem 1]{K08} another criterion for the essential self-adjointness of $(L, C_0^{\infty}(\mathbb{R}^d))$, which is stated as follows. If the condition {\bf ($\gamma$)} below holds, then $(L, C_0^{\infty}(\mathbb{R}^d))$ is essentially self-adjoint. \\
\\
{\it
{\bf ($\gamma$)}{\rm :} $\rho>0$ a.e. with $\sqrt{\rho} \in H^{1,2}_{loc}(\mathbb{R}^d)$ and $\left \|\frac{\nabla \rho}{\rho}\right\| \in L^4_{loc}(\mathbb{R}^d, \mu)$. For any $R>0$, there exists constants $\varepsilon>0$ and $C>0$ such that
$$
\sup_{B \subset B_{R}}\left( \frac{1}{dx(B)} \int_{B} \rho(x)^{1+\varepsilon} dx\right) \left( \frac{1}{dx(B)} \int_{B} \rho(x)^{-1-\varepsilon} dx\right) \leq C.
$$
}
\\
Remarkably, regardless of dimension $d \geq 1$, if $\rho$ is locally bounded and locally unifomly positive and $\|\nabla\rho\| \in L^{4}_{loc}(\mathbb{R}^d)$, then the condition {\bf ($\gamma$)} is fulfilled, so that $(L, C_0^{\infty}(\mathbb{R}^d))$ is essentially self-adjoint by \cite[Theorem 1]{K08}. Thus, \cite[Theorem 1]{K08} is more general than \cite[Theorem 1]{L99}, if $\rho$ is locally bounded and locally unifomly positive. However, if $d=2 \text{ or } 3$, then the condition that $\rho$ is locally bounded and locally unifomly positive and that $\|\nabla\rho\| \in L^{4}_{loc}(\mathbb{R}^d)$ is still more restrictive than ($\alpha^\prime$). Hence, in case where $d=2$ or $3$, it is natural to raise a question of whether the essential self-adjointness of $(L, C_0^{\infty}(\mathbb{R}^d))$ holds or not if $\rho$ is locally bounded and locally uniformly positive and $\| \nabla \rho \| \in L^d_{loc}(\mathbb{R}^d)$. 
Before presenting our main result, let us introduce our basic two conditions, which will be assumed individually. \\ \\
{
{\bf (H)}{\rm:} \it
$d \geq 3$, $\rho \in H^{1,d}_{loc}(\mathbb{R}^d)$, and for any open ball $B$ in $\mathbb{R}^d$, there exists a constant $c_B \in (0,1)$ such that
\begin{equation} \label{bddpos}
c_{B} \leq \rho(x) \leq c^{-1}_{B}, \quad \text{ for a.e. $x \in B$.} 
\end{equation}
} \\
{
{\bf (S)}{\rm:} \it
$d \geq 3$ and there exist $N_0 \in \mathbb{N}$ and $\varepsilon>0$ such that $\rho \in C(\mathbb{R}^d) \cap H^{1,d}_{loc}(\mathbb{R}^d \setminus B_{N_0}) \cap H_{loc}^{1,2+\varepsilon}(\mathbb{R}^d)$ and for any open ball $B$ in $\mathbb{R}^d$, there exists a constant $c_B \in (0,1)$ such that \eqref{bddpos} holds.
} \\ \\
Then our main result is Theorem \ref{mainthm} which is stated as below.
\begin{theo} \label{mainthm}
Assume that either {\bf (H)}  or {\bf (S)} holds. Then $(L, C_0^{\infty}(\mathbb{R}^d))$ defined as in \eqref{dirichletop} is $L^r(\mathbb{R}^d, \mu)$-unique for any $r \in (1, 2]$. In particular,   $(L, C_0^{\infty}(\mathbb{R}^d))$ is essentially self-adjoint.
\end{theo}
\text{}\\
Theorem \ref{mainthm} under the assumption {\bf (H)} with $d=3$ and Theorem \ref{mainthm} under the assumption {\bf (S)} are not covered by the literature mentioned above. The main ingredient to show Theorem \ref{mainthm} is the following elliptic regularity result Lemma \ref{mainlem}, which is a direct consequence of Theorem \ref{maintheo1}.
\text{}\\
\begin{lem} \label{mainlem}
Assume that either {\bf (H)}  or {\bf (S)} holds. Let $(L, C_0^{\infty}(\mathbb{R}^d))$ be defined as in \eqref{dirichletop}. If $h \in L_{loc}^2(\mathbb{R}^d)$ satisfies
$$
\int_{\mathbb{R}^d} (1-L)u \cdot h d\mu=0, \quad \forall u \in C_0^{\infty}(\mathbb{R}^d),
$$
then $h \in H^{1,2}_{loc}(\mathbb{R}^d) \cap C(\mathbb{R}^d)$.
\end{lem}
\text{}\\
We mention that under a weaker condition on $h$ than $h \in L^2_{loc}(\mathbb{R}^d)$ (precisely, under a condition on $h d\mu$ replaced by a locally finite measure $\nu$ satisfying $Lf \in L^1(\mathbb{R}^d, \nu)$ for all $f \in C_0^{\infty}(\mathbb{R}^d)$) and under the assumption {\bf ($\alpha^{\prime}$)}, a similar result to  Lemma \ref{mainlem} follows from \cite[Theorem 1]{BKR97}. But \cite[Theorem 1]{BKR97} does not imply Lemma \ref{mainlem}, since the conditions {\bf (H)} and {\bf (S)} are weaker than {($\alpha^{\prime}$)}.\\
In order to prove Lemma \ref{mainlem} (precisely, Theorem \ref{maintheo1}), we ultimately show that for an arbitrarily  given $r>0$, $h$ is the weak limit of a sequence of functions in $H^{1,2}(B_{r/2})$ by using the Banach–Alaoglu theorem, so that we obtain $h \in H^{1,2}_{loc}(\mathbb{R}^d)$. This idea originates from \cite[Theorem 2.1]{St99} (cf. \cite[Theorem 2.20]{LST22}), where Schauder regularity results on $\mathbb{R}^d$ (\cite[Theorems 4.3.1, 4.3.2]{Kr96}) are crucially used. But in contrast to \cite[Theorem 2.1]{St99}, we consider the resolvent operator on an open ball $B_r$ and a recent result for well-posedness for $H_0^{1,s}(B_r)$-solutions of elliptic equations of divergence type with $s \in (2, \infty)$ (\cite[Corollary 2.1.6]{BKRS}). 
Once $h \in H^{1,2}_{loc}(\mathbb{R}^d)$ is shown, we directly obtain $h \in C(\mathbb{R}^d)$ from integration by parts and a $C^{0, \gamma}$-regularity result  as in \cite[Corollary 5.5]{T73}. The rest of the proof of Theorem 1.1 is similar to the one of \cite[Theorem 2.3]{Eb}. To underline the novelty of our result, we provide various examples where the conditions {\bf($\alpha^{\prime}$)}, {\bf ($\beta$)} and {\bf ($\gamma$)} are not satisfied, but either {\bf (H)} or {\bf (S)} is fulfilled.\\
We raise an open problem of whether the essential self-adjointness of $(L, C_0^{\infty}(\mathbb{R}^d))$ holds or not under the assumption {\bf (H)} where $d \geq 3$ is replaced by $d=2$.
\section{Preliminaries}
We refer to \cite[Notations and Conventions]{LST22} for basic notations which are not defined in this paper. 
Let $s \in [1, \infty)$ and $U$ be an open subset of $\mathbb{R}^d$. Define a Sobolev space
$$
H^{2,s}(U) := \{f \in L^s(U): \partial_i f \in L^s(U), \, \partial_i \partial_j f \in L^s(U)  \text{ for all $1 \leq i,j \leq d$} \}
$$
equipped with the norm
$$
\|f\|_{H^{2,s}(U)} = \left(\int_U |f|^s dx  + \sum_{i=1}^d \int_U |\partial_i f|^s dx + \sum_{i,j=1}^d \int_U  |\partial_i \partial_j f|^s dx \right)^{\frac{1}{s}}.
$$
Denote by $L_{loc}^s(\overline{U})$ the set of all Lebesgue measurable functions $f$ on $\overline{U}$ for which $f \in L^s(V)$ for any bounded open subset $V$ in $\mathbb{R}^d$ with $\overline{V} \subset \overline{U}$. Define $H^{1,s}_{loc}(\overline{U}):=\{f \in L^s_{loc}(\overline{U}): \partial_i f \in L^s_{loc}(\overline{U}) \text{ for all $i=1, \ldots, d$} \}$. For $q \in [1, \infty]$ and $\bold{F} \in L^q(U, \mathbb{R}^d)$, we write 
$\| \bold{F} \|_{L^q(U)}:=\| \bold{F} \|_{L^q(U, \mathbb{R}^d)}$.
For $r>0$, denote by $B_r$ (resp. $B_r(z)$) the open ball in $\mathbb{R}^d$ centered at the origin (resp. $z$) with radius $r$. For $f \in L^1_{loc}(\mathbb{R}^d)$, we write $f \in VMO$ if there exists a positive continuous function $\omega$ on $[0, \infty)$ with $\omega(0)=0$ such that
\begin{equation} \label{vmoine}
\sup_{z \in \mathbb{R}^d, r <R} r^{-2d} \int_{B_r(z)} \int_{B_r(z)} |f(x)-f(y)|dx dy \leq \omega(R), \quad \forall R>0.
\end{equation}
For an integrable function $f$ and a bounded Borel measurable subset $A$ of $\mathbb{R}^d$, we write $f_{A} := \frac{1}{dx(A)}\int_{A} f dx$. For a function $g$ on a bounded open subset $U$ of $\mathbb{R}^d$, we write $g \in VMO(U)$ if there exists an extension of $g$ on $\mathbb{R}^d$, say again $g$, such that $g \in VMO$. The following are well-known facts (see \cite[the first paragraph of page 8]{BKRS}), but we provide a proof for the convenience of the reader.
\begin{prop} \label{vmoexte}
Let $U$ be a bounded open subset of $\mathbb{R}^d$.
\begin{itemize}
\item[(i)]
If $U$ has a Lipschitz boundary and $f \in H^{1,d}(U)$, then $f \in VMO(U)$.
\item[(ii)] 
If $f \in C(\overline{U})$, then $f \in VMO(U)$.
\end{itemize}
\end{prop}
\begin{proof}
(i) First, by \cite[Theorem 4.7]{EG15} there exists an extension of $f$ on $\mathbb{R}^d$, say again $f$, such that $f \in H^{1,d}(\mathbb{R}^d)_0$. 
Let $z \in \mathbb{R}^d$ and $R>0$ be given and choose $r \in (0, R)$. Then by the Poincar\'{e} inequality (\cite[Theorem 4.9]{EG15}), there exists a constant $C_d>0$ such that
\begin{eqnarray*}
\int_{B_r(z)} |f(x)-f_{B_r(z)}|dx &\leq& C_d \cdot r \int_{B_r(z)} \| \nabla f\| dx \leq  C_d \cdot r\cdot dx(B_r(z))^{1-\frac{1}{d}}\left(\int_{B_r(z)}\| \nabla f\|^d dx \right)^{1/d} \\
&=& C_d \cdot dx(B_1)^{1-\frac{1}{d}} \, \cdot r^d  \left(\int_{B_r(z)}\| \nabla f\|^d dx \right)^{1/d}.
\end{eqnarray*}
Thus,
\begin{eqnarray*}
\hspace{-0.5em}\int_{B_r(z)} \int_{B_r(z)} |f(x)-f(y)|dx dy 
&\leq& 2 C_d dx(B_1)^{2-\frac{1}{d}} \cdot r^{2d} \left(\int_{B_r(z)}\| \nabla f\|^d dx \right)^{1/d}.
\end{eqnarray*}
Now define $\varphi(z,t) := 2 C_{d} \cdot dx(B_1)^{2-\frac{1}{d}} \cdot  \left(\int_{B_{t}(z)} \| \nabla f\|^d dx\right)^{1/d}$, $(z,t) \in \mathbb{R}^d \times [0, \infty)$. Then, $|\varphi|$ is bounded by $2 C_{d} \cdot dx(B_1)^{2-\frac{1}{d}} \cdot \| \nabla f\|_{L^d(\mathbb{R}^d)}$ on $\mathbb{R}^d \times [0, \infty)$. We can hence define
$$
\omega(t):= \sup_{z \in \mathbb{R}^d} \varphi(z, t), \quad t \in [0, \infty).
$$
Then $\omega$ is positive with $\omega(0)=0$ and \eqref{vmoine} holds. Now it remains to show that $\omega \in C([0, \infty))$. Note that $\varphi$ is continuous on $\mathbb{R}^d \times [0, \infty)$ by Lebesgue's Theorem, hence $\varphi$ is uniformly continous on any compact subset of $\mathbb{R}^d \times [0, \infty)$. Moreover, since $f$ has a compact support in $\mathbb{R}^d$, for any $\delta>0$ there exists a compact subset $V$ of $\mathbb{R}^d$ such that
$$
\omega(t)= \sup_{z \in V} \varphi(z, t), \quad t \in [0, \delta].
$$
Therefore, it follows that $\omega$ is continuous on $[0, \delta]$ as desired. \\
(ii) By the Tietze extension theorem, there exists an extension of $f$, say again $f$, such that $f \in C_0(\mathbb{R}^d)$. Thus, $f$ is uniformly continuous on $\mathbb{R}^d$. Let 
$$
\omega(t) = dx(B_1)^2 \cdot \sup_{z \in \mathbb{R}^d} \sup_{(x, y) \in \overline{B}_t(z) \times \overline{B}_t(z) } |f(x)-f(y)|.
$$
Then $\omega \in C([0, \infty))$ with $\omega(0)=0$ as desired.
\end{proof}
\\ \\
 For basic concepts about Dirichlet forms, resolvents, generators and semigroups, we refer to \cite[Chapter I and II. 1, 2]{MR} and \cite[Chapter 1]{FOT}. Assume that either {\bf (H)} or {\bf (S)} holds. Then $\frac{1}{\rho} \in L^1_{loc}(\mathbb{R}^d)$. Hence, by \cite[Chapter II. 2 a), d)]{MR} the bilinear form defined by
$$
\mathcal{E}^{0, B_{r}}(f,g) = \int_{B_{r}} \langle \nabla f, \nabla g  \rangle d\mu, \quad f, g \in C_0^{\infty}(B_{r})
$$
is closable on $L^2(B_{r}, \mu)$ and its closure $(\mathcal{E}^{0, B_{r}}, D(\mathcal{E}^{0, B_{r}}))$ is a symmetric Dirichlet form. Let 
$$
\mathcal{E}_1^{0, B_r}(f,g):=\mathcal{E}^{0, B_r}(f,g)+\int_{B_r} f g d\mu, \quad f,g \in D(\mathcal{E}^{0, B_r}).
$$
Then $D(\mathcal{E}^{0, B_{r}})$ equipped with the inner product $\mathcal{E}_1^{0, B_r}(\cdot, \cdot)$, say $(D(\mathcal{E}^{0, B_{r}}),\mathcal{E}_1^{0, B_r})$, is a Hilbert space whose norm is equivalent to $\|\cdot \|_{H^{1,2}_0(B_r)}$ and $D(\mathcal{E}^{0, B_{r}})=H^{1,2}_0(B_r)$, since $\rho$ is bounded below and above by strictly positive constants on $B_r$. Denote the corresponding sub-Markovian $C_0$-resolvent of contractions on $L^2(B_r, \mu)$ by  $(G^{r}_{\alpha})_{\alpha>0}$ and the corresponding generator on $L^2(\mathbb{R}^d, \mu)$ by $(L^{0, B_r}, D(L^{0, B_{r}}))$.

\section{Elliptic regularity results}
\begin{lem} \label{reglema}
Assume that either {\bf (H)} or {\bf (S)} holds and let $r>0$.  If $f \in L^{\infty}(B_r)$ and $\alpha>0$, then $G^r_{\alpha} f \in H^{1,s}_0(B_r) \cap H^{2, 2}(B_r)$ for any $s \in [2, \infty)$.
\end{lem}
\begin{proof}
Since $\rho$ is bounded below and above by strictly positive constants on $B_r$, we have $D(\mathcal{E}^{0, B_r})=H_0^{1,2}(B_r)$. Let $r>0$, $f \in L^{\infty}(B_r)$ and $\alpha>0$. By \cite[I. Theorem 2.8]{MR}, $G_{\alpha} f \in D(\mathcal{E}^{0, B_r})=H_0^{1,2}(B_r)$ and
\begin{eqnarray} 
&&\int_{B_r}  \langle \rho \nabla G^r_{\alpha} f, \nabla \varphi   \rangle dx  +  \int_{B_r} \alpha \rho G_{\alpha}^r f \cdot \varphi dx  =\int_{B_r} \langle \nabla G^r_{\alpha} f, \nabla \varphi   \rangle d\mu  + \alpha \int_{B_r} G_{\alpha}^r f \cdot \varphi d\mu  \nonumber \\
&&=\mathcal{E}^{0,B_r}(G^r_{\alpha} f, \varphi) + \alpha \int_{B_r} G^{r}_{\alpha}f \cdot \varphi d\mu= \int_{B_r} f \cdot \varphi d\mu = \int_{B_r} \rho f \cdot \varphi dx, \qquad \forall \varphi \in C_0^{\infty}(B_r). \quad \qquad \label{variationaleq}
\end{eqnarray}
Note that $\rho \in VMO(B_r)$ by Proposition \ref{vmoexte} and that $0<\inf_{B_r} \rho \leq \sup_{B_r} \rho<\infty$ by the assumption. Since $\rho f \in L^{\infty}(B_r)$, it follows from \cite[Theorem 2.1.8, Corollary 2.1.6]{BKRS} that $G_{\alpha}^r f \in H_0^{1, s}(B_r)$ for any $s \in [2, \infty)$. Let
$$
g:= f+\Big \langle \frac{\nabla \rho}{\rho}, \nabla G^r_{\alpha} f \Big \rangle- \alpha G_{\alpha}^r f.
$$
Since $\rho \in H^{1,2+\varepsilon}_{loc}(\mathbb{R}^d)$ for some $\varepsilon>0$, we have $g \in L^{2}(B_r)$. Moreover, \eqref{variationaleq} is equivalent to
\begin{eqnarray} \label{variationaleq2}
&&\int_{B_r} \langle \rho \nabla G^r_{\alpha} f, \nabla \varphi   \rangle dx  +\int_{B_r} \Big \langle \nabla \rho, \nabla G^r_{\alpha} f \Big \rangle \varphi dx  = \int_{B_r}\rho g \cdot \varphi dx, \quad \forall \varphi \in C_0^{\infty}(B_r).
\end{eqnarray}
On the other hand, by \cite[Theorem 9.15]{Gilbarg} there exists $u \in H^{1,2}_0(B_r) \cap H^{2, 2}(B_r)$ such that
$$
- \Delta u = g, \quad \text{a.e. on $B_r$. }
$$
Applying integration by parts to the above, we then obtain
\begin{equation} \label{diveq}
\int_{B_r} \langle \rho \nabla u, \nabla \varphi \rangle dx  +\int_{B_r} \langle \nabla \rho, \nabla u \rangle  \varphi dx = \int_{B_r} \rho g\cdot \varphi dx.
\end{equation}
Using the maximum principle (\cite[Theorem 2.1.8]{BKRS}) with \eqref{variationaleq2} and \eqref{diveq}, we get $u=G_{\alpha}^r f$ in $H^{1,2}_0(B_r)$, hence $G_{\alpha}^r f \in H^{2,2}(B_r)$ as desired.
\end{proof}

\begin{prop} \label{prop:0.2}
Assume that either {\bf (H)} or {\bf (S)} holds and let $r>0$.  If $u \in H^{1,s}_0(B_r) \cap H^{2,2}(B_r)$ for any $s \in [2, \infty)$, then $u \in D(L^{0,B_r})$ and
$$
L^{0, B_r} u =\Delta u+ \Big \langle \frac{\nabla \rho}{\rho}, \nabla u \Big \rangle.
$$
\end{prop}
\begin{proof}
Given $u \in H^{1,2}_0(B_r) \cap H^{2,2}(B_r)$ and $\varphi \in C_0^{\infty}(B_r)$, it follows from integration by parts that
$$
\mathcal{E}^{0, B_r}(u, \varphi)=\int_{B_r} \langle \nabla u, \nabla \varphi \rangle d\mu= -\int_{B_r} \left( \Delta u +\Big \langle \frac{\nabla \rho}{\rho}, \nabla u \Big \rangle \right) \varphi d\mu
$$
Since $u \in H^{2,2}(B_r)$, $\|\nabla u\| \in L^s(B_r)$ for any $s \in [2, \infty)$ and $\|\nabla \rho\| \in L^{2+\varepsilon}(B_r)$ for some $\varepsilon>0$, we have 
$$
\Delta u + \Big \langle \frac{\nabla \rho}{\rho}, \nabla u \Big \rangle \in L^2(B_r, \mu).
$$
Thus, the assertion holds by \cite[I. Proposition 2.16]{MR}.
\end{proof}

\begin{theo} \label{maintheo1}
Assume that either {\bf (H)} or {\bf (S)} holds. Let $(L, C_0^{\infty}(\mathbb{R}^d))$ be defined as in \eqref{dirichletop}. If $c \in L^d_{loc}(\mathbb{R}^d)$, $h \in L^2_{loc}(\mathbb{R}^d)$ and
\begin{equation} \label{distributionalde}
\int_{\mathbb{R}^d}(c-L) u \cdot h d\mu =0, \quad \forall u \in C_0^{\infty}(\mathbb{R}^d), 
\end{equation}
is satisfied, then $h \in H^{1,2}_{loc}(\mathbb{R}^d) \cap C(\mathbb{R}^d)$. Moreover, $h \in H^{1,s}_{loc}(\mathbb{R}^d)$ for all $s \in (2, \infty)$.
\end{theo}
\begin{proof}
If we assume {\bf (H)}, then choose an arbitrary $N_0 \in \mathbb{N}$. On the other hand, if we assume {\bf (S)}, then choose $N_0 \in \mathbb{N}$ as in {\bf(S)}. To show $h \in H^{1,2}_{loc}(\mathbb{R}^d)$, it suffices to show that given $r>2N_0$, $h \in H^{1,2}(B_{r/2})$. Let $r>2N_0$ and $v \in H^{1,s}_0(B_r) \cap
 H^{2,2}(B_r)$ for any $s \in [2, \infty)$ with $\text{supp}(v) \subset B_{r}$. Then, by mollification, there exists $(v_n)_{n \geq 1} \subset C_0^{\infty}(B_r)$ such that 
$$
\lim_{n \rightarrow \infty} v_n =v \;\; \text{ in $H^{2,2}(B_r)$} \quad \text{ and } \quad \lim_{n \rightarrow \infty} v_n =v \;\; \text{ in } H^{1, s}(B_r) \quad \text{ for any $s \in [2, \infty)$. }
$$
Since $\| \nabla \rho \| \in L^{2+\varepsilon}(B_r)$ for some $\varepsilon>0$ and $0<\inf_{B_r} \rho \leq \sup_{B_r} \rho <\infty$, it follows from Proposition \ref{prop:0.2} and \eqref{distributionalde} that we have $v \in D(L^{0, B_r})$ and that
\begin{equation} \label{mainequ}
\int_{B_r} \left( cv -L^{0,B_r} v \right) h d\mu=  \int_{B_r}\left( cv-\Big \langle \frac{\nabla \rho}{\rho}, \nabla v \Big \rangle - \Delta v \right) h d\mu=0.
\end{equation}
Let $f_n \in L^{\infty}(B_r)$ be such that $\lim_{n \rightarrow \infty}f_n = h$ in $L^2(B_{r}, \mu)$ with
$\|f_n\|_{L^2(B_r, \mu)} \leq \|h\|_{L^2(B_r, \mu)}$ for all $n \in \mathbb{N}$. Let $\chi \in C_0^{2}(B_{r}) \subset  D(L^{0, B_r})_{b}$ with $\chi =1$ on $B_{r/2}$. Given $\alpha>0$, it holds $G^r_{\alpha} f_n \in D(L^{0, B_r})_b$, hence it follows from \cite[Theorem 4.2.1]{BH} that $\chi  G^{r}_{\alpha} f_n \in D(L^{0, B_r})_{b} \subset D(\mathcal{E}^{0, B_r})$ and that
\begin{eqnarray}
L^{0, B_r}(\chi G^{r}_{\alpha} f_n) &=&  G^{r}_{\alpha} f_n \cdot L^{0, B_r} \chi+ 2\langle \nabla \chi, \nabla G^{r}_{\alpha}f_n \rangle + \chi \cdot L^{0, B_r} G^{r}_{\alpha} f_n  \label{algebraiden} \\
&=&   G^{r}_{\alpha} f_n \cdot L^{0, B_r} \chi + 2\langle \nabla \chi, \nabla G^{r}_{\alpha}f_n \rangle +(\chi \alpha G_{\alpha}^r f_n -\chi f_n). \nonumber
\end{eqnarray}
Therefore,
\begin{eqnarray}
&&\mathcal{E}^{0, B_r} \left( \chi \alpha G^r_{\alpha} f_n, \chi \alpha G^r_{\alpha} f_n   \right)  = -\int_{B_r} L^{0, B_r}(\chi \alpha G^r_{\alpha} f_n) \cdot \chi \alpha G^r_{\alpha} f_n   d\mu \nonumber \\
&&\;\;= - \int_{B_r} \alpha G^r_{\alpha} f_n \cdot \left(L^{0, B_r} \chi \right) \cdot \chi \alpha G^r_{\alpha} f_n d\mu -2\int_{B_r} \langle \nabla \chi, \nabla \alpha G^{r}_{\alpha}f_n \rangle \chi \alpha G^r_{\alpha} f_n d\mu \nonumber \\
&& \;\;\; \quad -\alpha \int_{B_r} (\chi \alpha G_{\alpha}^r f_n -\chi f_n) \chi \alpha G^r_{\alpha} f_n d\mu  \nonumber  \\
&& \;\; =  - \int_{B_r} \alpha G^r_{\alpha} f_n \cdot \left(L^{0, B_r} \chi \right) \cdot \chi \alpha G^r_{\alpha} f_n d\mu  -2\int_{B_r} \langle  \nabla \chi, \nabla(\chi \alpha G_{\alpha}^r f_n)  \rangle \alpha G^r_{\alpha} f_n d\mu  \nonumber \\
&& \qquad \quad +2\int_{B_r} \langle  \nabla \chi, \nabla \chi  \rangle (\alpha G^r_{\alpha} f_n)^2 d\mu  -\alpha \int_{B_r} (\chi \alpha G_{\alpha}^r f_n -\chi f_n) \chi \alpha G^r_{\alpha} f_n d\mu.  \label{weakconverest}
\end{eqnarray}
Note that by Proposition \ref{prop:0.2}, $\chi \in D(L^{0, B_r})$ and  $L^{0, B_r} \chi =\Delta \chi + \langle \frac{\nabla \rho}{\rho}, \nabla \chi \rangle$. Since $L^{0, B_r} \chi =0$ on $B_{\frac{r}{2}}$, $\|\frac{\nabla \rho}{\rho}\| \in L^d(B_r \setminus \overline{B}_{N_0})$ and $\frac{r}{2}>N_0$, we obtain $L^{0, B_r} \chi \in L^d(B_r, \mu)$. Using the H\"{o}lder inequality, $L^2(B_r, \mu)$-contraction property of $(G^{r}_{\alpha})_{\alpha>0}$ and the Sobolev inequality, we obtain
\begin{eqnarray}
&&- \int_{B_r} \alpha G^r_{\alpha} f_n \cdot \left(L^{0, B_r} \chi \right) \cdot \chi \alpha G^r_{\alpha} f_n d\mu  \nonumber \\
&&\;\; \quad \leq \| \alpha G^r_{\alpha} f_n \|_{L^2(B_r, \mu)} \| L^{0, B_r} \chi \|_{L^d(B_r, \mu)} \| \chi \alpha G^r_{\alpha} f_n  \|_{L^{\frac{2d}{d-2}}(B_r, \mu)} \nonumber \\
&&\;\; \quad \leq \underbrace{ \|h\|_{L^2(B_r, \mu)} \| L^{0, B_r} \chi \|_{L^d(B_r, \mu)}}_{=:C_1} \cdot K_{d,r, \rho}  \, \mathcal{E}^{0, B_r}(\chi \alpha G^r_{\alpha} f_n, \chi \alpha G^r_{\alpha} f_n  )^{1/2}, \label{estim1}
\end{eqnarray}
where $K_{d,r, \rho}>0$ is a constant which only depends on $d$, $r$ and $\rho$. Using  the H\"{o}lder inequality and $L^2(B_r, \mu)$-contraction property of $(G^{r}_{\alpha})_{\alpha>0}$,
 we get
\begin{eqnarray}
&& -2 \int_{B_r} \langle  \nabla \chi, \nabla(\chi \alpha G_{\alpha}^r f_n)  \rangle \alpha G^r_{\alpha} f_n d\mu \nonumber \\
&& \quad \leq 2\, \underbrace{\| \nabla \chi \|_{L^{\infty}(B_r)} \|h\|_{L^2(B_r, \mu)}}_{=:C_2}  \,  \mathcal{E}^{0, B_r}(\chi \alpha G^r_{\alpha} f_n, \chi \alpha G^r_{\alpha} f_n  )^{1/2}. \label{estim2}
\end{eqnarray}
Therefore, we obtain
\begin{eqnarray*}
\mathcal{E}^{0, B_r} \left( \chi \alpha G^r_{\alpha} f_n, \chi \alpha G^r_{\alpha} f_n   \right)  \leq C_3 \, \mathcal{E}^{0, B_r} \left( \chi \alpha G^r_{\alpha} f_n, \chi \alpha G^r_{\alpha} f_n   \right)^{1/2}+C_4,
\end{eqnarray*}
where $C_3=C_1 K_{d,r, \rho} +2 C_2$ and  $C_4=2C_2^2 +2\alpha \| \chi\|^2_{\infty} \|h\|^2_{L^2(B_r, \mu)}$. Hence, it follows by Young's inequality that
$$
\sup_{n \geq 1} \mathcal{E}^{0, B_r} \left( \chi \alpha G^r_{\alpha} f_n, \chi \alpha G^r_{\alpha} f_n   \right) \leq C_3^2 +2C_4.
$$
By using $L^2(\mathbb{R}^d, \mu)$-contraction propoerty of $(G^r_{\alpha})_{\alpha>0}$ and the Banach-Alaoglu Theorem, there exists a subsequence of $(f_n)_{n \geq 1}$, say again $(f_n)_{n \geq 1}$, such that
\begin{equation} \label{weakconve1}
\lim_{n \rightarrow \infty} \chi \alpha G^r_{\alpha}f_n =   \chi \alpha G^r_{\alpha}h, \quad \text{ weakly in $(D(\mathcal{E}^{0, B_r}), \mathcal{E}_1^{0, B_r})$.}
\end{equation}
Meanwhile, note that $(\alpha-L^{0, B_r})G_{\alpha}^r f_n =f_n$ and that by Lemma \ref{reglema},
$\chi^2 G^r_{\alpha} f_n \in H^{1,s}_0(B_r) \cap
H^{2,2}(B_r)$ for any $s \in [2, \infty)$ with $\text{supp}(\chi^2 G^r_{\alpha} f_n ) \subset B_r$. Thus, we obtain from \eqref{algebraiden} and \eqref{mainequ} that
\begin{eqnarray}
&& - \alpha\int_{B_r} (\alpha G_{\alpha}^r f_n -f_n) h\chi^2 d\mu = \int_{B_r}- L^{0, B_r} (\alpha G^r_{\alpha} f_n) \cdot h \chi^2  \,d\mu \nonumber \\
&& \quad = \int_{B_r}  - L^{0, B_r} (\chi^2 \alpha G^r_{\alpha} f_n) \cdot h d\mu+ \int_{B_r}  \alpha G^{r}_{\alpha} f_n \cdot (L^{0, B_r} \chi^2) h d\mu + 2\int_{B_r} \langle \nabla \chi^2, \nabla \alpha G^{r}_{\alpha}f_n \rangle h d\mu \nonumber \\
&&\quad = \int_{B_r}-c \chi^2 \alpha G^r_{\alpha} f_n \cdot h d\mu+ \int_{B_r}  \alpha G^{r}_{\alpha} f_n \cdot (L^{0, B_r} \chi^2) h d\mu + 2\int_{B_r} \langle \nabla \chi^2, \nabla \alpha G^{r}_{\alpha}f_n \rangle h d\mu. \label{resesteq}
\end{eqnarray}
Since 
$$
-\alpha \int_{B_r} (\alpha G^r_{\alpha} h - h)^2 \chi^2 d\mu \leq 0,
$$
it follows by \eqref{resesteq}, \eqref{weakconve1} and the Sobolev inequality that
\begin{eqnarray}
&&-\alpha \int_{B_r} (\alpha G^r_{\alpha}h - h) \alpha G^r_{\alpha} h \cdot \chi^2 d\mu \leq- \alpha \int_{B_r} (\alpha G^r_{\alpha} h -h) h \chi^2 d\mu  \nonumber \\
&& \quad \quad = \lim_{n \rightarrow \infty}  - \alpha \int_{B_r} (\alpha G^r_{\alpha} f_n -f_n) h \chi^2 d\mu \nonumber  \\
&&\quad \quad  = \lim_{n \rightarrow \infty}  \left( -\int_{B_r}c \chi^2 \alpha G^r_{\alpha} f_n \cdot h d\mu+ \int_{B_r}  \alpha G^{r}_{\alpha} f_n \cdot (L^{0, B_r} \chi^2) h d\mu + 2\int_{B_r} \langle \nabla \chi^2, \nabla \alpha G^{r}_{\alpha}f_n \rangle h d\mu \right) \nonumber  \\
&&\quad \quad  = -\int_{B_r}c \chi^2 \alpha G^r_{\alpha} h \cdot h d\mu+ \int_{B_r}  \alpha G^{r}_{\alpha} h \cdot (L^{0, B_r} \chi^2) h d\mu + \int_{B_r} \langle \nabla \chi^2, \nabla \alpha G^{r}_{\alpha} h \rangle h d\mu \nonumber  \\
&& \quad \quad \leq   C_5 \mathcal{E}^{0, B_r} (\chi \alpha G^r_{\alpha} h, \chi \alpha G^r_{\alpha} h)^{1/2}, \qquad \qquad \label{fundestimated}
\end{eqnarray}
where $C_5=\left(  K_{d,r, \rho}\|\chi c\|_{L^d(B_r)} +   K_{d,r, \rho}\|L^{0, B_r} \chi^2\|_{L^d(B_r)} + \|\nabla \chi^2 \|_{L^\infty(B_r)} \right) \|h\|_{L^2(B_r, \mu)}$. By using \eqref{weakconve1}, \eqref{weakconverest}, \eqref{fundestimated} and applying the estimates \eqref{estim1}, \eqref{estim2} to the case where $f_n$ is replaced by $h$,  it follows
\begin{eqnarray*}
&&\mathcal{E}^0(\chi \alpha G^r_{\alpha}h, \chi \alpha G^r_{\alpha}h ) \leq \liminf_{n \rightarrow \infty} \mathcal{E}^0(\chi \alpha G^r_{\alpha}f_n, \chi \alpha G^r_{\alpha}f_n ) \\
&& \qquad  =- \int_{B_r} \alpha G^r_{\alpha} h \cdot \left(L^{0, B_r} \chi \right) \cdot \chi \alpha G^r_{\alpha} h d\mu -2\int_{B_r} \langle  \nabla \chi, \nabla(\chi \alpha G_{\alpha}^r h)  \rangle \alpha G^r_{\alpha} h d\mu \\
&& \qquad \qquad+ 2\int_{B_r} \langle  \nabla \chi, \nabla \chi  \rangle (\alpha G^r_{\alpha} h)^2 d\mu  -\alpha \int_{B_r} (\chi \alpha G_{\alpha}^r h-\chi h) \chi \alpha G^r_{\alpha}h  d\mu   \\
&& \qquad \qquad \leq (C_3+C_5) \mathcal{E}^{0, B_r}(\chi \alpha G^r_{\alpha}h, \chi \alpha G^r_{\alpha}h)^{1/2} +2C^2_2.
\end{eqnarray*}
Thus by Young's inequality, it holds
$$
\sup_{\alpha>0} \mathcal{E}^0(\chi \alpha G^r_{\alpha}h, \chi \alpha G^r_{\alpha}h ) \leq (C_3+C_5)^2+4C_2^2
$$
Finally, using the Banach-Alaoglu Theorem and the strong continuity of $(G^r_{\alpha})_{\alpha>0}$ in $L^2(B_r, \mu)$, we obtain $\chi h \in D(\mathcal{E}^{0, B_r})$, so that $h \in H^{1,2}(B_{r/2})$ as desired. \\
Finally using integration by parts, we conclude that \eqref{distributionalde} implies that
\begin{equation} \label{basicdiveq}
\int_{\mathbb{R}^d} \langle \rho \nabla h, \nabla \varphi \rangle dx +\int_{\mathbb{R}^d} c \rho h \cdot \varphi dx=0, \quad \varphi \in C_0^{\infty}(\mathbb{R}^d),
\end{equation}
hence $h$ is locally H\"{o}lder continuous on $\mathbb{R}^d$ by \cite[Corollary 5.5]{T73}.  Moreover, applying \cite[Theorem 1.8.3]{BKRS} to \eqref{basicdiveq} with Proposition \ref{vmoexte}, we obtain $h \in H^{1,s}_{loc}(\mathbb{R}^d)$ for all $s \in (2, \infty)$.
\end{proof}  

\section{Strong uniqueness}

\begin{proof}{\bf of Theorem \ref{mainthm}}\;\; Given $r \in (1, 2]$, choose  $q \in [2, \infty)$ so that $\frac{1}{r}+\frac{1}{q}=1$. Let $h \in L^q(\mathbb{R}^d, \mu)$ be such that
\begin{equation*} 
\int_{\mathbb{R}^d} (1-L)u \cdot h\, d\mu = 0, \qquad \forall u \in C_0^{\infty}(\mathbb{R}^d). 
\end{equation*} 
In order to show $L^r(\mathbb{R}^d, \mu)$-uniqueness of $(L, C_0^{\infty}(\mathbb{R}^d))$, it is enough to show $h=0$, $\mu$-a.e. (cf. \cite[Remark 3.4]{LT21in}).  By Theorem \ref{maintheo1} and integration by parts and an approximation, we obtain $h \in H^{1,2}_{loc}(\mathbb{R}^d) \cap C(\mathbb{R}^d)$ and
\begin{equation} \label{underlying}
\int_{\mathbb{R}^d} \langle \rho \nabla h, \nabla \varphi \rangle dx +\int_{\mathbb{R}^d} \rho h \cdot \varphi dx=0, \quad \varphi \in H^{1,2}(\mathbb{R}^d)_0.
\end{equation}
Since the map $\psi: \mathbb{R} \rightarrow \mathbb{R}$ defined by $\psi(t)= t|t|^{q-2}$, $t \in \mathbb{R}$ satisfies $\psi \in C^1(\mathbb{R})$ with $\psi'(t)=(q-1)|t|^{q-2}$ for all $t \in \mathbb{R}$ and $h \in C(\mathbb{R}^d)$, it holds by the chain rule (cf. \cite[Theorem 4.4(ii)]{EG15}), $h|h|^{q-2} \in H^{1,2}_{loc}(\mathbb{R}^d) \cap C(\mathbb{R}^d)$ and 
$$
\nabla(h |h|^{q-2}) = (q-1) |h|^{q-2}\nabla h.
$$
Let $v \in H^{1, \infty}(\mathbb{R}^d)_0$. By replacing $\varphi$ in \eqref{underlying} with $v h |h|^{q-2} \in H^{1,2}_0(\mathbb{R}^d, \mu)_0$, it follows that
\begin{eqnarray} \label{inteineq}
\int_{\mathbb{R}^d} v|h|^q d\mu &=&-(q-1)\int_{\mathbb{R}^d} \langle \nabla h, \nabla h \rangle |h|^{q-2} vd\mu  - \int_{\mathbb{R}^d} \langle \nabla v, \nabla h \rangle h |h|^{q-2} d\mu 
\end{eqnarray}
Choose $v = w^2$, where $w \in H^{1,\infty}(\mathbb{R}^d)_0$. Then \eqref{inteineq} and  the Cauchy-Schwarz inequality and Young's inequality ($|a b| \leq \frac{q-1}{4}{a^2} +\frac{1}{(q-1)}b^2$)
implies
\begin{eqnarray*}
&&\int_{\mathbb{R}^d} w^2|h|^q d\mu + (q-1)\int_{\mathbb{R}^d} \|\nabla h\|^2 |h|^{q-2} w^2 d\mu  \leq  - 2\int_{\mathbb{R}^d} \langle \nabla w, \nabla h \rangle wh|h|^{q-2} d\mu  \\
&&\leq2  \left(\int_{\mathbb{R}^d} \|\nabla h\|^2 |h|^{q-2} w^2d\mu   \right)^{1/2} \left(\int_{\mathbb{R}^d} \|\nabla w\|^2   |h|^{q} d\mu \right)^{1/2} \\
&& \leq \frac{q-1}{2} \int_{\mathbb{R}^d} \|\nabla h\|^2 |h|^{q-2} w^2d\mu +\frac{2}{(q-1)} \int_{\mathbb{R}^d} \|\nabla w\|^2   |h|^{q} d\mu,
\end{eqnarray*}
hence it follows that
\begin{equation} \label{mainineq}
\int_{\mathbb{R}^d} w^2|h|^q d\mu  \leq \frac{2}{(q-1)} \int_{\mathbb{R}^d} \|\nabla w\|^2   |h|^{q} d\mu.
\end{equation}
Given $k \in \mathbb{N}$, we define
$$
w_k(x):= \min\Big( ( k+1-\|x\|  )^+, 1 \Big), \quad x \in \mathbb{R}^d.
$$
Then $w_k \in H^{1, \infty}(\mathbb{R}^d)_0 \cap C(\mathbb{R}^d)$ with $\|\nabla w_k\|_{L^{\infty}(\mathbb{R}^d)} \leq 1$ and $w_k =1$ on $B_k$. By substituting $w_k$ for $w$ in \eqref{mainineq}, we obtain
\begin{equation*}
\int_{B_k} |h|^q d\mu \leq \int_{\mathbb{R}^d} w_k^2|h|^q d\mu  \leq \frac{2}{(q-1)} \int_{\mathbb{R}^d} \|\nabla w_k\|^2   |h|^{q} d\mu.
\end{equation*}
Since $\lim_{k \rightarrow \infty} \|\nabla w_k \|=0$, $\mu$-a.e., the above implies $\int_{\mathbb{R}^d} |h|^q d\mu= \lim_{k \rightarrow \infty} \int_{B_k} |h|^q d\mu = 0$ by  Lebesgue's theorem. Therefore, $h=0$ in $L^q(\mathbb{R}^d, \mu)$, so that $(L, C_0^{\infty}(\mathbb{R}^d))$ is $L^r(\mathbb{R}^d, \mu)$-unique. Since the essential self-adjointness of $(L, C_0^{\infty}(\mathbb{R}^d))$ is equivalent to $L^2(\mathbb{R}^d, \mu)$-uniqueness of $(L, C_0^{\infty}(\mathbb{R}^d))$ by \cite[Lemma 1.4]{Eb}, the assertion follows.
\end{proof}

\begin{rem}
\begin{itemize}
\item[(i)]
The assumption {\bf (H)} does not guarantee $L^1(\mathbb{R}^d, \mu)$-uniqueness of $(L, C_0^{\infty}(\mathbb{R}^d))$. Indeed, if $\rho(x)=e^{\|x\|^{2+\varepsilon}}$, $x \in \mathbb{R}^d$ for some $\varepsilon>0$, then it is known that the corresponding semigroup generated by the closure of $(L, C_0^{\infty}(\mathbb{R}^d))$ in $L^2(\mathbb{R}^d, \mu)$ is not conservative, hence it follows from \cite[Corollary 2.3]{St99} that $(L, C_0^{\infty}(\mathbb{R}^d))$ is not $L^1(\mathbb{R}^d, \mu)$-unique. 
\item[(ii)]
Assume that {\bf ($\alpha^{\prime}$)} (in Section \ref{intro}) with $d=2$ holds. Analogously to the proofs of Theorems \ref{maintheo1} and \ref{mainthm}, we can show that $(L, C_0^{\infty}(\mathbb{R}^d))$ defined as in \eqref{dirichletop} is $L^r(\mathbb{R}^d, \mu)$-unique for any $r \in (1, 2]$ and $(L, C_0^{\infty}(\mathbb{R}^d))$ is essentially self-adjoint. But we point out that under the assumption {\bf ($\alpha^{\prime}$)} with $d=2$, the essential self-adjointness of $(L, C_0^{\infty}(\mathbb{R}^d))$ was already shown in \cite[Theorem 7]{BKR97}.
\item[(iii)]
In case $d=1$, it does not seem easy to generalize  \cite[2.3 Theorem]{W85} where the essential self-adjointness of $(L, C_0^{\infty}(\mathbb{R}))$ defined as in \eqref{dirichletop} is shown under the assumption that $\rho \in H^{1,2}_{loc}(\mathbb{R}) \cap C(\mathbb{R})$ with $\rho(x)>0$ for all $x \in \mathbb{R}$.
\end{itemize}
\end{rem}
\text{}\\
In the following examples, we will present a specific $\rho$ for which either {\bf (H)} or {\bf (S)} holds, but {\bf ($\alpha^{\prime}$)}, {\bf ($\beta$)} and {\bf ($\gamma$)} in Section \ref{intro} are not fulfilled.
\begin{exam}
\begin{itemize}
\item[(i)]
Define
$$
\rho(x):= \left\{\begin{matrix}
1, && \quad x =0 \\
1+\frac{-1}{\ln \|x\|}, &&\quad x \in B_{\frac12} \\ 
1+\frac{2}{\ln 2} \|x\|, &&\quad x \in \mathbb{R}^d \setminus B_{\frac12}
\end{matrix}\right.
$$
Then $\rho \in C(\mathbb{R}^d) \cap C^{\infty}(\mathbb{R}^d \setminus \overline{B}_{1/2})$ with $\rho(x) \geq 1$ for all $x \in \mathbb{R}^d$, hence \eqref{bddpos} holds. On the other hand, we claim that $\rho \in H^{1,d}_{loc}(\mathbb{R}^d)$ but $\rho \notin \cup_{p \in (d, \infty)}H_{loc}^{1,p}(\mathbb{R}^d)$. To show the claim, let
$$
\psi(t):= \left\{\begin{matrix}
1, && \quad t =0 \\
1+\frac{-1}{\ln t}, &&\quad t \in (0,\frac12) \\ 
1+\frac{2}{\ln 2} t, &&\quad t \in [\frac12, \infty)
\end{matrix}\right.
$$
Then it holds $\psi \in C([0, \infty))$ and $\rho = \psi(\|x\|)$, hence $\rho \in C(\mathbb{R}^d)$. Note that by the chain rule, 
$$
\nabla \rho(x):= \left\{\begin{matrix}
\frac{1}{\|x\|(\ln \|x\|)^2\,} \frac{x}{\|x\|}  &&\quad \text{ for a.e. } x \in B_{\frac12} \\[5pt]
\frac{2}{\ln 2} \frac{x}{\|x\|}, &&\quad \text{  for a.e. } x \in \mathbb{R}^d \setminus B_{\frac12}.
\end{matrix}\right.
$$
Indeed, let $\beta \in [0, \infty)$ be a constant. Then,
\begin{eqnarray*}
&&\int_{B_{\frac12}}\| \nabla \rho \|^{d+\beta} dx = \int_0^{\frac12} \int_{\partial B_r} \left( \frac{1}{\|x\|(\ln \|x\|)^2\,} \right)^{d+\beta} dS \,dr \\
&&= c_{d} \int_0^{\frac12} \frac{1}{r^{\beta}(\ln r)^{2d+2\beta} } \, \frac{1}{r} dr  = c_d \int_{-\infty}^{-\ln2} \frac{1}{t^{2d+2\beta}} \frac{1}{e^{\beta t}} dt \\
&& = c_d \int_{\ln2}^{\infty} \frac{e^{\beta t}}{t^{2d+2\beta}} dt <\infty, \qquad \text{(where $c_d = dx(B_1) \cdot d$)},
\end{eqnarray*}
if and only if $\beta=0$. Thus, $\|\nabla \rho\| \in L^d(B_{1/2})$ and $\|\nabla \rho\| \notin \cup_{p \in (d, \infty)}L^{p}(B_{1/2})$. Therefore, both {\bf (H)} and {\bf (S)} hold, but {\bf ($\alpha^{\prime}$)} is not satisfied. Moreover in case of $d=3$, both {\bf ($\beta$)} and {\bf ($\gamma$)} are not satisfied.
\item[(ii)]
Define
\begin{equation*} \label{w1dvmo}
\rho(x):=2+\|x\|^2+\cos \left( \ln \ln \left(1+\frac{1}{\|x\|} \right)  \right), \quad {x \in \mathbb{R}^d \setminus \{ 0\}, \quad \rho(0)=0.}
\end{equation*}
Then it is easy to check that \eqref{bddpos} holds. Since $\lim_{\|x\| \rightarrow 0} \rho (x)$ does not exist, we have $\rho \notin \cup_{p \in (d, \infty)}  H^{1,p}_{loc}(\mathbb{R}^d)$ by Morrey's embedding. Now we claim that $\rho \in H^{1,d}_{loc}(\mathbb{R}^d)$.  Let
$$
v(x):=\ln \ln \left(1+\frac{1}{\|x\|} \right), \quad x \in \mathbb{R}^d \setminus \{0\}.
$$
Then $v \in C^{\infty}(\mathbb{R}^d \setminus \overline{B}_{1/2})$. To show the claim, it suffices to show by the chain rule that $v \in H^{1,d}(B_1)$. First, note that 
$$
\int_{B_1} |v|^d dx  = c_d \int_0^1 \left( \ln \ln \left(1+\frac{1}{r}\right) \right)^d r^{d-1} dr <\infty,
$$
where $c_d = dx(B_1) \cdot d$ since the integrand of the right hand side above is bounded on $(0, \varepsilon)$ for some $\varepsilon \in (0,1)$. 
Furthermore, by the chain rule we have
$$
\nabla v(x)  = \frac{1}{\ln (1+\frac{1}{\|x\|})} \cdot \frac{1}{1+\frac{1}{\|x\|}} \frac{-1}{\|x\|^2} \frac{x}{\|x\|}, \quad \text{for a.e. $x \in B_1$}
$$
and
$$
\int_{B_1} \|\nabla v\|^d dx  = c_d \int_0^1 \frac{1}{ \left(\ln(1+\frac{1}{r})\right)^d \cdot (1+\frac{1}{r})^d \cdot r^{d+1}  } dr<\infty,
$$
since there exist constants $M>0$ and $\varepsilon \in (0,\frac12)$ such that the integrand of the right hand side above is bounded by $\frac{M}{r |\ln r|^d}$ on $(0, \varepsilon)$ and that
$$
\int_0^{\varepsilon} \frac{1}{r |\ln r|^d} dr <\infty.
$$
Hence the claim is shown. Therefore, {\bf (H)} holds, but  neither {\bf (S)} nor {\bf ($\alpha^{\prime}$)} holds. In particular, if $d=3$, then both {\bf ($\beta$)} and {\bf ($\gamma$)} are not satisfied.
\item[(iii)]
Let $d=3$. Given $\varepsilon \in (0,\frac{1}{2})$, define for $x=(x_1, x_2, x_3) \in \mathbb{R}^3$
$$
\rho_1(x):=\left\{\begin{matrix}
 \frac{|x_1|}{(x_1^2+x_2^2)^{\frac{1}{2+\varepsilon}}}, & \quad \text{ if }\; x_1^2+x_2^2>0.
   \\[10pt]
 0, & \quad \text{ if }\; (x_1, x_2)=0.
\end{matrix}\right.
$$
Then $\rho_1 \in C(\mathbb{R}^3)$ and $\partial_3 \rho_1=0$, Moreover, for $i \in \{1, 2 \}$ it holds
\begin{eqnarray*}
\partial_i \rho_1  = \frac{\delta_{1i} \frac{x_1}{|x_1|} (x_1^2+x_2^2)^{\frac{1}{2+\varepsilon}} -\frac{|x_1|}{2+\varepsilon} \cdot \frac{2x_i}{(x_1^2+x_2^2)^{1-\frac{1}{2+\varepsilon}}}}{(x_1^2+x_2^2)^{\frac{2}{2+\varepsilon}}}  = \frac{\delta_{1i}}{(x_1^2+x_2^2)^{\frac{1}{2+\varepsilon}}} \frac{x_1}{|x_1|}-\frac{|x_1|}{2+\varepsilon} \frac{2x_i}{(x_1^2+x_2^2)^{1+\frac{1}{2+\varepsilon}}}, \
\end{eqnarray*}
where $\delta_{11}=1$ and $\delta_{12}=0$. Thus, it follows from the triangle inequality that
$$
|\partial_1 \rho_1|  \geq \frac{1 -\frac{2}{2+\varepsilon}}{(x_1^2+x_2^2)^{\frac{1}{2+\varepsilon}}} \;\text{ and }\; \;  \|\nabla \rho_1\| \leq \sqrt{2}\left(1+\frac{2}{2+\varepsilon}\right) \frac{1}{(x_1^2+x_2^2)^{\frac{1}{2+\varepsilon}}} \text{ a.e. on $\mathbb{R}^3$}.
$$
By Fubini's theorem and the change of variables, we obtain
\begin{eqnarray*}
&&\int_{-\sqrt{\frac12}}^{\sqrt{\frac12}} \int_{\{ x_1^2+x_2^2 \leq \frac12 \}} |\partial_1 \rho_1|^{2+\varepsilon} dx_1 dx_2 dx_3 \geq 2\pi \left( 1-\frac{2}{2+\varepsilon} \right)^{2+\varepsilon} \int_{-\sqrt{\frac12}}^{\sqrt{\frac12}} \int_0^{\sqrt{\frac12}} \frac{1}{r} dr dx_3 = \infty,
\end{eqnarray*}
but for any $\lambda \in [0, \varepsilon)$ and $R>0$ we have
\begin{eqnarray*}
\;\int_{-R}^{R} \int_{\{ x_1^2+x_2^2 \leq R^2 \}} \|\nabla \rho_1 \|^{2+\lambda} dx_1 dx_2 dx_3 \leq k \int_{-R}^{R}
\int_{0}^{R} \frac{1}{r^{2\left( \frac{2+\lambda}{2+\varepsilon} \right)-1}} dr dx_3<\infty,
\end{eqnarray*}
where $k=\sqrt{2}^{\lambda} \left( 1+\frac{2}{2+\varepsilon}  \right)^{\lambda}$. Hence $\rho_1 \in H_{loc}^{1,2+\lambda}(\mathbb{R}^3) \cap C(\mathbb{R}^3)$ for any $\lambda \in [0, \varepsilon)$, but $\|\nabla \rho_1\| \notin L^{2+\varepsilon}(B_1)$. Choose $\zeta \in C_0^{\infty}(\mathbb{R}^3)$ so that $\zeta \geq 0$ on $B_2$,  $\zeta = 1$ on $B_{1}$ and ${\rm supp}(\zeta) \subset B_2$. Now define
$$
\rho(x):=1+\|x\|^2 + \zeta(x) \rho_1(x), \quad x \in \mathbb{R}^3.
$$
Then $\rho \in  H^{1,2+\lambda}_{loc}(\mathbb{R}^3) \cap C(\mathbb{R}^3) \cap C^{\infty}(\mathbb{R}^3 \setminus \overline{B}_2)$ for all $\lambda \in [0, \varepsilon)$, but  $\|\nabla \rho \| \notin L^{2+\varepsilon}(B_1)$. Therefore, {\bf ($\alpha^{\prime}$)}, {\bf ($\beta$)}, {\bf ($\gamma$)} and {\bf (H)} are not satisfied, but {\bf (S)} is fulfilled.
\item[(iv)]
Define
$$
\phi(t) :=\left\{\begin{matrix}
\displaystyle \frac{1}{(\ln t)^2} \cdot\frac{1}{|t|}, \quad &t \in (-\frac12,\frac12) \\[1em]
0, \quad \quad  &t \in \mathbb{R} \setminus (-\frac12, \frac12)
\end{matrix}\right. 
$$
Then it follows from integration by substitution that $\phi \in L^1\left( \mathbb{R}\right)$, but $\phi \notin L^{1+\theta}((-\frac12, \frac12))$ for any $\theta>0$. Now let $\varepsilon \in (0, 1)$ and define
$$
\psi(t) = \phi(t)^{\frac{1}{2+\varepsilon}}, \qquad t \in (-1,1).
$$
Then it holds $\psi \in L^{2+\varepsilon}(\mathbb{R})$, but $\psi \notin L^{2+\varepsilon+\delta}((-1,1))$ for any $\delta>0$.
Now define
$$
v(t) = \int_{-\infty}^{t} \psi(s) ds, \qquad t \in \mathbb{R}.
$$
Then $v \geq 0$ on $\mathbb{R}$ and it follows from \cite[Lemma 8.2]{BRE} that $v \in C(\mathbb{R}) \cap H^{1,2+\varepsilon}_{loc}(\mathbb{R})$ and $v'=\psi$ on $\mathbb{R}$. Choose $\zeta \in C_0^{\infty}(\mathbb{R})$ so that $\zeta \geq 0$ on $\mathbb{R}$, $\zeta=1$ on $(-1,1)$ and ${\rm supp}(\zeta) \subset (-2,2)$. Define
$$
\rho(x) = 1+\|x\|^2+ \prod_{i=1}^d \zeta(x_i)v(x_i), \qquad x = (x_1, \ldots, x_d) \in \mathbb{R}^d.
$$
Then $\rho(x) \geq 1$ for any $x \in \mathbb{R}^d$ and $\rho \in H^{1,2+\varepsilon}_{loc}(\mathbb{R}^d) \cap C(\mathbb{R}^d) \cap C^{\infty}(\mathbb{R}^d \setminus \overline{B}_{2\sqrt{d}})$, but
$\|\nabla \rho \| \notin L^{2+\varepsilon+\delta}(B_{\sqrt{d}})$ for any $\delta>0$. Therefore, {\bf ($\alpha^{\prime}$)}, {\bf ($\beta$)}, {\bf ($\gamma$)} and {\bf (H)} are not satisfied, but {\bf (S)} is fulfilled.
\end{itemize}
\end{exam}
\text{}\\
{\bf Acknowledgement.} \; The author would like to thank Professor Gerald Trutnau for his valuable comments and suggestions. Additionally, the author wishes to extend his deepest appreciation to the anonymous referee for providing helpful comments and suggestions on this manuscript.

\text{}\\

\centerline{}
Haesung Lee\\
Department of Mathematics and Computer Science, \\
Korea Science Academy of KAIST, \\
105-47 Baegyanggwanmun-ro, Busanjin-gu,\\
Busan 47162, Republic of Korea\\
E-mail: fthslt14@gmail.com

\begin{thebibliography}{XXX}

\bibitem{AHS77}
S. Albeverio,  R. Høegh-Krohn, L. Streit, {\it Energy forms, Hamiltonians, and distorted Brownian paths},
J. Mathematical Phys. 18 (1977), no. 5, 907–917.


\bibitem{BKR97} V. I. Bogachev, N. V. Krylov, M. R\"{o}ckner, {\it Elliptic regularity and essential self-adjointness of Dirichlet operators on $\mathbb{R}^n$}, Ann. Scuola Norm. Sup. Pisa Cl. Sci. (4) 24 (1997), no. 3, 451–461.


\bibitem{BKRS} V. I. Bogachev, N. V. Krylov, M. R\"{o}ckner, and S. V. Shaposhnikov, {\it  Fokker-Planck-Kolmogorov equations}, Mathematical Surveys and Monographs, 207. American Mathematical Society, Providence, RI, 2015.


\bibitem{BH}
N. Bouleau, F. Hirsch, {\it Dirichlet forms and analysis on Wiener space}, De Gruyter Studies in Mathematics, 14. Walter de Gruyter Berlin, 1991.


\bibitem{BRE} H. Brezis, {\it
Functional analysis, Sobolev spaces and partial differential equations,} 
Universitext. Springer, New York, 2011.


\bibitem{CF96}
P. Cattiaux, M. Fradon, {\it Entropy, reversible diffusion processes, and Markov uniqueness},
J. Funct. Anal. 138 (1996), no. 1, 243–272.



\bibitem{Eb}
 A. Eberle, {\it Uniqueness and non-uniqueness of semigroups generated by singular diffusion operators}, Lecture Notes in Mathematics, 1718. Springer-Verlag, Berlin (1999).


\bibitem{Eb20} 
A. Eberle,  {\it $L^p$ uniqueness of non-symmetric diffusion operators with singular drift coefficients. I. The finite-dimensional case}, J. Funct. Anal. 173 (2000), no. 2, 328–342.

\bibitem{EG15} L.C. Evans, R.F. Gariepy, 
{\it Measure theory and fine properties of functions},
Revised edition. Textbooks in Mathematics. CRC Press, Boca Raton, FL, 2015.

\bibitem{FOT}
M. Fukushima, Y. Oshima, M. Takeda, {\it Dirichlet forms and symmetric Markov processes}, Second revised and extended edition, De Gruyter Studies in Mathematics, 19. Walter de Gruyter \& Co., Berlin, 2011.


\bibitem{Gilbarg}
D. Gilbarg, N.S. Trudinger, {\it Elliptic partial differential equations of second order}, 
Reprint of the 1998 edition. Classics in Mathematics. Springer-Verlag, Berlin, 2001.


\bibitem{K08}
M. Kolb, {\it On the strong uniqueness of some finite dimensional Dirichlet operators}, Infin. Dimens. Anal. Quantum Probab. Relat. Top. 11 (2008), no. 2, 279–293.

\bibitem{Kr96}
N. V. Krylov, {\it Lectures on elliptic and parabolic equations in H\"{o}lder spaces}, Graduate Studies in Mathematics, 12. American Mathematical Society, Providence, RI, 1996.



\bibitem{LS92}
V. Liskevich, Y. A. Semenov, {\it Dirichlet operators: A priori estimates and the
uniqueness problem}, J. Funct. Anal. 109 (1992), no. 1, 199–213.


\bibitem{L99}
V. Liskevich, {\it On the uniqueness problem for Dirichlet operators}, J. Funct. Anal. 162 (1999), no. 1, 1–13. 


\bibitem{LST22}
H. Lee, W. Stannat, G. Trutnau, \textit{Analytic theory of Itô-stochastic differential equations with non-smooth coefficients}, SpringerBriefs in Probability and Mathematical Statistics. Springer, Singapore, 2022. 


\bibitem{LT21in}
H. Lee, G. Trutnau, {\it Existence and uniqueness of (infinitesimally) invariant measures for second order partial differential operators on Euclidean space}, J. Math. Anal. Appl. 507 (2022), no. 1, Paper No. 125778.




\bibitem{MR}
Z. Ma, M. Röckner, {\it Introduction to the theory of (non-symmetric) Dirichlet forms}, Universitext. Springer-Verlag, Berlin, 1992. 



\bibitem{RZ94}
M. R\"{o}ckner, T. S. Zhang, {\it Uniqueness of generalized Schrödinger operators. II}, J. Funct. Anal. 119 (1994).


\bibitem{St99} 
W. Stannat, {\it (Nonsymmetric) Dirichlet operators on $L^1$: Existence, uniqueness and associated Markov processes},  Ann. Scuola Norm. Sup. Pisa Cl. Sci. (4) 28. 1999. No. 1. 99-140.




\bibitem{T92}
M. Takeda, {\it The maximum Markovian selfadjoint extensions of generalized Schrödinger operators}, J. Math. Soc. Japan 44 (1992), no. 1, 113–130. 


\bibitem{T73}
N. S. Trudinger, {\it Linear elliptic operators with measurable coefficients}, Ann. Scuola Norm. Sup. Pisa (3) 27 (1973), 265--308.


\bibitem{W85}
N. Wielens, {\it The essential self-adjointness of generalized Schr\"{o}dinger
operators}, J. Funct. Anal. 65 (1985), 98-115.

\text{}\\


\end{thebibliography}
\end{document}